\documentclass[11pt]{amsart}

\usepackage{amsmath,amssymb,amsthm}
\usepackage{ifthen,verbatim}
\usepackage{mathrsfs}
\usepackage{color}
\usepackage{url}
\usepackage{pdfpages}
\usepackage{subfigure}
\usepackage[english]{babel}
\usepackage{verbatim,here}
\usepackage[T1]{fontenc}
\usepackage{floatflt,graphicx,graphics}
\usepackage[urlcolor=blue,colorlinks=true]{hyperref}
\usepackage{url}
\usepackage{tikz}
\usepackage{tabularx}

\newtheorem{theorem}{Theorem}[section]
\newtheorem{corollary}[theorem]{Corollary}
\newtheorem{lemma}[theorem]{Lemma}
\newtheorem{proposition}[theorem]{Proposition}
\theoremstyle{definition}

\newtheorem{example}[theorem]{Example}
\newtheorem{remark}[theorem]{Remark}
\newtheorem{nonsec}[equation]{}

\numberwithin{equation}{section}

\newcommand{{\tth}}{\mathrm{th}}

\newcommand{\K}{\mathcal{K}}

\font\fFt=eusm10 %scaled 1200
\font\fFa=eusm7  %scaled 1200
\font\fFp=eusm5  %scaled 1200
\def\K{\mathchoice
{\hbox{\,\fFt K}}
{\hbox{\,\fFt K}}
{\hbox{\,\fFa K}}
{\hbox{\,\fFp K}}}
\title[Landen transformations]{Landen transformations applied to approximation}

\author[R. Kargar]{Rahim Kargar}

\address[R. Kargar (Corresponding Author)]{Department of Mathematics and Statistics, University of Turku, Turku, Finland}
\email{{\tt rakarg@utu.fi}}

\author[O. Rainio]{Oona Rainio}
\address[O. Rainio]{Department of Mathematics and Statistics, University of Turku, Turku, Finland}
\email{\tt ormrai@utu.fi}

\author[M. Vuorinen]{Matti Vuorinen}
\address[M. Vuorinen]{Department of Mathematics and Statistics, University of Turku, Turku, Finland}
\email{\tt vuorinen@utu.fi}



\begin{document}

\begin{abstract}
We study computational methods for the approximation of special
functions recurrent in geometric function theory and quasiconformal mapping
theory. The functions studied can be expressed as quotients of complete
elliptic integrals and as inverses of such quotients. In particular, we
consider the distortion function $\varphi_K(r)$ which gives a majorant
for $|f(x)|$ when $f: \mathbb{B}^2 \to \mathbb{B}^2, f(0)=0,$ is a
quasiconformal mapping of the unit disk $\mathbb{B}^2.$ It turns out that
the approximation method is very simple: five steps of Landen iteration is
enough to achieve machine precision.
\end{abstract}

\maketitle
\hspace{2.4cm} {In memoriam: Academician Yu. G. Reshetnyak 1929-2021}

\medskip

%%%%%%%%%%%%%%%%%%%%%%%%%%%%%%%%%%%%%%%%%%%%%%%%%%%%%%%%%%%%%%%%%%%%%%%%%%%%%%%%

\section{Introduction}

The {\it Landen transformation }
\begin{equation*}
 r \mapsto \frac{2 \sqrt{r}}{1+r},\quad r \in (0,1),
\end{equation*}
and the {\it Landen sequences} of functions, recursively defined in terms of
this transformation, are closely related to elliptic integrals and elliptic
functions. For instance, the complete integrals $\K(r)$ satisfy functional
identities involving the Landen transformation, and these integrals can be
expressed as infinite products, where the factor functions are expressed
in terms of the Landen sequences.
%~~~~~~~~~~~~~~~~~~~~~~~~~~~~~~~~~~~~~~~~~~~~

Our goal here is to show that some of the well-known special
functions of geometric function theory can be efficiently
computed using a few steps of the Landen iteration.
These functions include the function $\mu(r)$ related to the
conformal modulus of the Gr\"{o}tzsch ring domain, defined as a quotient
of complete elliptic integrals,
and
\begin{equation*}
  \varphi_K:[0,1]\rightarrow [0,1];\quad \varphi_K(r)=\mu^{-1}\left(\mu(r)/K\right),\quad K>0,
\end{equation*}
the special function in the quasiconformal Schwarz lemma. The paper \cite{avz}
is a survey of these special functions. However, this survey is incomplete because
of the very extensive work of the authors of \cite{qmc} after the publication of
\cite{avz}.
%~~~~~~~~~~~~~~~~~~~~~~~~~~~~~~~~~~~~~~~~~~~~

In recent years, many papers have been published with information
about the asymptotic behavior of the function $\K(r)$ at the
singularity $r=1$, see Ref. \cite[p. 53]{avv}. We give here a
recursive scheme for the numerical approximation for the several functions we
consider. In five iteration steps, we obtain approximations with errors close to machine epsilon. The code is just a few lines and can be implemented in every programming language --- no programming libraries are needed. For instance, the first iteration for the function $ \varphi_K$ improves the classical majorant function $4^{1-1/K} r^{1/K}$ for
$ \varphi_K(r), K > 1.$
%~~~~~~~~~~~~~~~~~~~~~~~~~~~~~~~~~~~~~~~~~~~~

The above functions have been extensively studied in the monograph \cite{avv} and
the associated software for computation is given on the attached diskette.
We use this
software as a reference when we study the precision of our algorithms.
The function $ \varphi_K(r)$ for integer values of $K$ also occurs in the study
of so-called modular equations \cite[p. 92]{avv}, \cite{aqvv, as}.
Previously, the Landen iteration applied to $ \varphi_K(r)$ has been
studied in \cite[p. 93]{avv} and in Partyka's paper \cite{par}.

The structure of this article is as follows: In Section \ref{sec2} we define the ascending
and descending Landen sequences and investigate their application to the
aforementioned approximation problems. In Section \ref{sec3} we analyze the algorithms in detail and study their numerical performance.

\section{Ascending and Descending Landen Sequences}\label{sec2}
%%%%%%%%%%%%%%%%%%%%%%%%%%%%%%
%%%%%%%%%%%%%%%%%%%%%%%%%%%%%%
%%%%%%%%%%%%%%%%%%%%%%%%%%%%%%
In this section, we review the Landen transformation and its applications to compute elliptic integrals and related special functions. These facts will be applied to quasiconformal mappings in the next section.
%~~~~~~~~~~~~~~~~~~~~~~~~~~~~~~~~~~~~~~~~~~~~
\begin{nonsec}{\bf Landen sequences.}\label{Landen sec}
For $r\in (0,1)$ let $L(r,0)=r$ and
\begin{equation}\label{Landen}
  L(r,p+1)=\frac{2\sqrt{L(r,p)}}{1+L(r,p)};\quad L(r,-p-1)=\left(\frac{L(r,-p)}{1+\sqrt{1-L(r,-p)^2}}\right)^2,
\end{equation}
for $p=0,1,2,3,\ldots$. The recursively defined sequences
$\{L(r,p)\}$ and $\{L(r,-p)\}$ are called {\it ascending} and
{\it descending Landen sequences}, respectively. It is clear that
each of  the Landen functions $L(\cdot,p):(0,1)\rightarrow (0,1)$
is an increasing homeomorphism with
\begin{equation*}
  L(r,p)<L(r,p+1) \quad {\rm and} \quad L(r,p+q)=L(L(r,p),q)
\end{equation*}
for all $r\in(0,1)$ and $p,q\in\mathbb{Z}:=\{0,\pm1,\pm2,\pm3,\ldots\}$.
In particular,
\begin{equation*}
  y=L(r,p)\Leftrightarrow r=L(y,-p).
\end{equation*}
Therefore, $L(\cdot,-p)$ is the inverse of $L(\cdot,p)$.
\end{nonsec}
%~~~~~~~~~~~~~~~~~~~~~~~~~~~~~~~~~~~~~~~~~~~~
Throughout the paper, we use the following abbreviations
\begin{equation*}
r' =\sqrt{1-r^2},\quad w=  1 -\frac{ r^4}{(1 + r')^4}=
\frac{4(2-r^2) r' -8(1-r^2)}{r^4}.
\end{equation*}
As Table \ref{LanFunc} suggests, the functions in the Landen sequence
become more involved when $|p|$ increases.
%~~~~~~~~~~~~~~~~~~~~~~~~~~~~~~~~~~~~~~~~~~~~
\begin{table}[H]
\centering
\begin{tabular}{|c|c|c|c|c|}
    \hline
 $L(r,-2)$ &   $L(r,-1)$ & $L(r,0)$ & $L(r,1)$ & $L(r,2)$ \\
   \hline
$\displaystyle{\frac{r^4}{(1+\sqrt{w})^2 (1+r')^4}}$ &$\displaystyle{\frac{r^2}{(1+r')^2}}$
& $r$     &$\displaystyle{\frac{2\sqrt{r}}{1+r}}$ &$\displaystyle{\frac{2\sqrt{2} r^{1/4}\sqrt{1+r}}{(1+\sqrt{r})^2}}$  \\
\hline
\end{tabular}
\vspace{0.3cm}
\caption{A few functions of the Landen sequence.}
\label{LanFunc}
\end{table}

%~~~~~~~~~~~~~~~~~~~~~~~~~~~~~~~~~~~~~~~~~~~~
\begin{proposition}
  Let $r\in(0,1)$ and $p=0,1,2,\ldots$. Then\\
  \noindent
  {\rm i)} \begin{equation}\label{two ineq for Landen}
  L(r,-p-1)<r^{2^{p}}\leq r;
\end{equation}
\noindent
{\rm ii)} \begin{equation}\label{lower bound for Landen}
      L(r,p+1)>r^{2^{-p}}\geq r.
    \end{equation}
\end{proposition}
%~~~~~~~~~~~~~~~~~~~~~~~~~~~~~~~~~~~~~~~~~~~~
\begin{proof}
 i) We prove the assertion by induction. It is clear that it holds for $p=0$. We assume that the inequality holds for $p=k-1$, i.e. $ L(r,-k)<r^{2^{k-1}}$. By the second identity of \eqref{Landen}, and using the last inequality, we get
\begin{equation*}
  L(r,-k-1)=\left(\frac{L(r,-k)}{1+\sqrt{1-L(r,-k)^2}}\right)^2\leq \frac{r^{2^k}}{\left(1+\sqrt{1-r^{2^k}}\right)^2}<r^{2^k}
\end{equation*}
for all $r\in(0,1)$, and $k>1$.\\
ii) Due to the similarity of the proof to i) we omit the details.
\end{proof}
%~~~~~~~~~~~~~~~~~
%~~~~~~~~~~~~~~~~~~~~~~~~~~~~~~~~~~~~~~~~~~~~
%~~~~~~~~~~~~~~~~~~~~~~~~~~~~~~~~~~~~~~~~~~~~
\begin{proposition}
  The following identities hold for all $r\in(0,1)$ and  $p=0,1,2,3,\ldots$.
  \begin{equation}\label{1st Iden}
    L(r,p)^2+L(r',-p)^2=1;
  \end{equation}
  \begin{equation}\label{2nd Iden}
    L(r,p-1)=\frac{1-L(r',-p)}{1+L(r',-p)}.
  \end{equation}
\end{proposition}
%~~~~~~~~~~~~~~~~~~~~~~~~~~~~~~~~~~~~~~~~~~~~
\begin{proof}
  In order to prove the first identity, we use induction. It is clear that the identity \eqref{1st Iden} holds true for $p=0$. Assume that \eqref{1st Iden} holds true when $p=k$, i.e.
  \begin{equation*}
    L(r,k)^2+L(r',-k)^2=1.
  \end{equation*}
  By using \eqref{Landen} and the last identity, we obtain
  \begin{align*}
    L(r,k+1)^2+L(r',-k-1)^2 &=\left(\frac{2\sqrt{L(r,k)}}{1+L(r,k)}\right)^2
    +\left(\frac{L(r',-k)}{1+\sqrt{1-L(r',-k)^2}}\right)^4 \\
    &=\frac{4L(r,k)(1+L(r,k))^2+(1-L(r,k)^2)^2}{(1+L(r,k))^4}=1
  \end{align*}
  concluding the proof of \eqref{1st Iden}.\\
The identity \eqref{2nd Iden} can be proved by applying \eqref{1st Iden}. We have
\begin{equation}\label{qu 1-L(rp p) d plus}
  \frac{1-L(r',-p)}{1+L(r',-p)}=\frac{1-\sqrt{1-L(r,p)^2}}{1+\sqrt{1-L(r,p)^2}}
=\left(\frac{1-\sqrt{1-L(r,p)^2}}{L(r,p)}\right)^2.
\end{equation}
On the other hand, if we replace $p$ by $-p$ in the second identity of \eqref{Landen} we get
\begin{equation}\label{L(r,p-1)}
  L(r,p-1)=\left(\frac{L(r,p)}{1+\sqrt{1-L(r,p)^2}}\right)^2.
\end{equation}
A simple calculation shows that the right-hand side of \eqref{qu 1-L(rp p) d plus} is equal to the last identity \eqref{L(r,p-1)}. The proof is now complete.
\end{proof}
%~~~~~~~~~~~~~~~~~~~~~~~~~~~~~~~~~~~~~~
\begin{remark}
  It follows from \eqref{1st Iden} and the first inequality of \eqref{two ineq for Landen} that
  \begin{equation*}
    L(r,p+1)=\sqrt{1-L(r',-p-1)^2}>\sqrt{1-(r')^{2^{p+1}}}:=\ell(r,p).
  \end{equation*}
  Computer experiment shows that this lower bound, i.e. $\ell(r,p)$, is better than the lower bound in \eqref{lower bound for Landen}, i.e. $r^{2^{-p}}$, when $r$ is close to one.
\end{remark}
%~~~~~~~~~~~~~~~~~~~~~~~~~~~~~~~~~~~~~~~~
The complete elliptic integral
\begin{equation*}
\K(r)=\int^1_0 \frac{dx}{\sqrt{(1-x^2)(1-r^2x^2)}},\quad
r\in(0,1),
\end{equation*}
defines a homeomorphism $\K :(0,1)\to (\pi/2, \infty).$
The following two identities due to Landen express important properties
of the complete elliptic integral $\K(r)$ \cite[p. 12]{borw}
(see also \cite[p. 51]{avv})
\begin{equation}\label{two iden K}
  \left\{
    \begin{array}{ll}
      \K\left(\displaystyle{\frac{2\sqrt{r}}{1+r}}\right)=(1+r)\K(r); \\\\
      \K\left(\displaystyle{\frac{1-r}{1+r}}\right)=\displaystyle{\frac{1}{2}(1+r)}\K(r'),
    \end{array}
  \right.
\end{equation}
%~~~~~~~~~~~~~~~~~~~~~~~~~~~~~~~~~~~~~~~~~~~~
The first identity of \eqref{two iden K} shows that
$\K(r) = \K(L(r,-1)) (1+ L(r,-1)).$ This observation
is the basis of the following classical result, see Ref. \cite[p. 14]{borw}.
Observe that $\K(L(r,-p)) \to \pi/2$ when $p \to \infty.$
\begin{lemma}\label{lem-K(r)=pi 2}
  For $r\in (0,1)$ we have
  \begin{equation*}
    \K(r)=\frac{\pi}{2}\prod_{n=0}^{\infty}\frac{2}{1+L(r',-n)}
    =\frac{\pi}{2}\prod_{n=1}^{\infty}(1+L(r,-n)).
  \end{equation*}
\end{lemma}
Lemma \ref{lem-K(r)=pi 2} gives fast
converging methods for numerical evaluation of $\K(r)$.
%~~~~~~~~~~~~~~~~~~~~~~~~~~~~~~~~~~~~~~~~~~
\begin{nonsec}{\bf Arithmetic-Geometric Mean.}\label{AG mean sec}
For $0<b<a$ define $a_0=a$, $b_0=b$, $a_{n+1}=(a_n+b_n)/2$ and $b_{n+1}=\sqrt{a_n b_n}$. The common limit of these sequences
\begin{equation*}
  AG(a,b)=\lim_{n\rightarrow\infty}a_n=\lim_{n\rightarrow\infty}b_n,
\end{equation*}
is the {\it arithmetic geometric mean} of $a$ and $b$.
An in-depth discussion of this is found in \cite{borw}, see
also Ref. \cite{bb93}. For the history of the arithmetic geometric mean,
we refer to \cite[pp. 17-27]{R}.

One of the basic properties of $AG(a,b)$ is that (see Ref. \cite[Lemma 4.3]{avv})
\begin{equation*}
  AG(a,b)=a AG(1, b/a),\quad a>0, b\geq0.
\end{equation*}
\end{nonsec}
%~~~~~~~~~~~~~~~~~~~~~~~~~~~~~~~~~~~~~~~~~~~~
\begin{theorem}\label{Thm-Gauss}
  {\rm(}Gauss{\rm)} For $r\in(0,1)$
  \begin{equation*}
    \K(r)=\frac{\pi}{2AG(1,r')}.
  \end{equation*}
\end{theorem}
From Lemma \ref{lem-K(r)=pi 2} and Theorem \ref{Thm-Gauss} we obtain the identity for $s\in(0,1)$
\begin{equation*}
  AG(1,s)=\prod_{n=0}^{\infty}\left(\frac{1+L(s,-n)}{2}\right).
\end{equation*}
%~~~~~~~~~~~~~~~~~~~~~~~~~~~~~~~~~~~~~~~~~~~~
\begin{lemma}
  Consider the arithmetic geometric mean iteration with $a=1$, $b\in(0,1)$, $\alpha\in[0,1]$, and let $b_n$ be the $n^{\rm th}$ iterate of the $b$-sequence. Then $b_n<L(b^\alpha,n)$ for $n=1,2,3\ldots$.
\end{lemma}
%~~~~~~~~~~~~~~~~~~~~~~~~~~~~~~~~~~~~~~~~~~~~
\begin{proof}
  We use induction to prove this lemma. We have
  \begin{equation*}
    b_1=\sqrt{a_0b_0}=\sqrt{b}\quad{\rm and }\quad L(b^\alpha,1)=\frac{2\sqrt{b^\alpha}}{1+b^\alpha}.
  \end{equation*}
  It is easy to see that $\sqrt{b}<2\sqrt{b^\alpha}/(1+b^\alpha)$ holds for all $b\in(0,1)$ and $\alpha\in[0,1]$. Let $b_k<L(b^\alpha,k)$ for all $k=n>1$. We need to show that $b_{k+1}<L(b^\alpha,k+1)$. Due to the fact that $a_k\in(0,1)$ for all $k>1$, and $t\mapsto 2\sqrt{t}/(1+t)$ is an increasing function, we get
  \begin{equation*}
    b_{k+1}=\sqrt{a_k b_k}<\sqrt{b_k}<\frac{2\sqrt{b_k}}{1+b_k}<\frac{2\sqrt{L(b^\alpha,k)}}{1+L(b^\alpha,k)}
    =L(b^\alpha, k+1),
  \end{equation*}
concluding the proof.
\end{proof}
%~~~~~~~~~~~~~~~~~~~~~~~~~~~~~~~~~~~~~~~~~~~~
\begin{remark}
  (1) There is a large body of literature about the properties of $\K(r)$ due, in particular, to the authors of \cite{qmc} and their students. See also the literature survey \cite{avz}.\\
  \noindent
  (2)  The {\it logarithmic mean} of $a,b>0, a \neq b$, is defined by (see \cite[p. 77]{avv})
  \begin{equation*}
  \mathcal{L}(a,b) = \frac{a-b}{\log(a/b)}=\prod_{k=1}^{\infty}\frac{a^{2^{-k}}+b^{2^{-k}}}{2}.
  \end{equation*}
  Denote $\mathcal{L}_t(a,b)= \mathcal{L}(a^t,b^t)^{1/t} $ for $t>0$. As shown in
\cite[Proposition 2.7]{bb93} the following very sharp inequality holds for $x\in(0,1)$
\begin{equation*}
\mathcal{L}_{3/2}(1,x)>AG(1,x) > \mathcal{L}(1,x).
\end{equation*}
\end{remark}
%~~~~~~~~~~~~~~~~~~~~~~~~~~~~~~~~~~~~~~~~~~~~
For what follows, the following decreasing homeomorphism
$\mu:(0,1)\to (0,\infty),$
\begin{equation}\label{muDef}
\mu(r)=\frac{\pi}{2}\frac{\K(r')}{\K(r)},\quad
r\in(0,1),
\end{equation}
is crucial. From \eqref{two iden K} we obtain \cite[p. 51]{avv}
\begin{equation}\label{two iden for mu}
  \mu\left(\frac{2\sqrt{r}}{1+r}\right)=\frac{1}{2}\mu(r);\quad
\mu\left(\frac{1-r'}{1+r'}\right)=2\mu(r).
\end{equation}
In terms of the Landen sequences we can write \eqref{two iden for mu} in the following form for $0<r<1$ and $p\in \mathbb{Z}$
\begin{equation}\label{mu(r)=2 p mu}
  \mu(r)=2^p \mu(L(r,p)).
\end{equation}
By \eqref{muDef} it is clear that
\begin{equation*}
  \mu(r) \mu(r')= \frac{\pi^2}{4}.
\end{equation*}
By Jacobi's work \cite[p. 462, (B.25)]{hkv} the inverse of $\mu$ can be expressed
in terms of theta-functions as follows for $y > 0$
\begin{equation*}
  \mu^{-1}(y) = \left(\frac{2 \sum^{\infty}_{n=0} q^{(n+1/2)^2}}{
  1+ 2 \sum^{\infty}_{n=1} q^{n^2}}\right)^2, \quad q= \exp(-2y).
\end{equation*}
Jacobi also proved formulas for $  \mu^{-1}(y)$ as infinite products,
see \cite[p. 91]{avv}.

Below we also study the special function
\begin{equation}\label{varphi}
  \varphi_K(r)=\mu^{-1}\left(\mu(r)/K\right),\quad r\in(0,1), K>0,
\end{equation}
It defines a homeomorphism $\varphi_K:[0,1]\rightarrow[0,1]$ with limit
values $\varphi_K(0)=0$ and $\varphi_K(1)=1$. The basic estimate for $ \varphi_K(r)$, $K\geq 1$, $r\in(0,1)$ is \cite[Thorem 10.9(1)]{avv}
\begin{equation}\label{esti varphi}
  r^{1/K}< \varphi_K(r)<4^{1-1/K} r^{1/K}.
\end{equation}
For information about this and other related inequalities the reader is referred
to \cite[p.319, 16.51]{hkv}.
By \eqref{mu(r)=2 p mu}
it is clear that for $r \in (0,1), p \in \mathbb{Z},$
\begin{equation}\label{varphi2}
 \varphi_{2^p}(r)= L(r,p).
\end{equation}
It is noteworthy that the functions $\mu, \mu^{-1}, \varphi_K$ satisfy many
functional identities \cite{avv0, avv}. For instance, the Pythagorean type
identity of the Landen transformation \eqref{1st Iden} has the following
counterparts for these functions \cite[p. 463, p. 125]{hkv}
\begin{equation*}
\left(\mu^{-1}(y)\right)^2 +\left(\mu^{-1}\left(\frac{\pi^2}{4y}\right)\right)^2 =1, \quad y>0;
\end{equation*}
\begin{equation}\label{phiPyth}
 \varphi_K(r)^2 + \varphi_{1/K}(r')^2 =1, \quad K>0, r \in (0,1).
\end{equation}
The following inequalities hold for $r\in (0,1)$ (see \cite[p. 122, (7.21)]{hkv}):
\begin{equation}\label{bounds for mu-1}
  \log \frac{1}{r}<\log \frac{1+3r'}{r}<\log \frac{(1+\sqrt{r'})^2}{r}
  <\mu(r)<\log\frac{2(1+r')}{r}<\log\frac{4}{r}.
\end{equation}
As a result of Jacobi's work, the following inequalities hold also for $0<r<1$ \cite[p. 91, (5.30)]{avv}
\begin{equation}\label{bounds for mu-2}
\log \frac{(1+\sqrt{r'})^2}{r}<{\rm arth} \sqrt[4]{r'}<\mu(r)<\frac{\pi^2}{4 {\rm arth}\sqrt[4]{r}}.
\end{equation}
We summarize the lower and upper bounds of $\mu(r)$ with their inverses in Tables \ref{Tab Lo b} and \ref{Tab up b}, respectively.
%~~~~~~~~~~~~~~~~~~~~~~~~~~~~~~~~~~~~~~~~~~~~~~~~
% Table for lower bounds and their inverses
%~~~~~~~~~~~~~~~~~~~~~~~~~~~~~~~~~~~~~~~~~~~~~~~~
\begin{table}[H]
\centering
\begin{tabular}{|c|c|c|}
    \hline
 $j$ & $u_j(r)$ &  $u_j^{-1}(y)$ \\
   \hline
  $1$ & $\displaystyle{{\rm arth}\sqrt[4]{r'}}$ & $\displaystyle{\sqrt{1-{\rm th}^8(y)}}$ \\
  \hline
  $2$ & $\displaystyle{\log \frac{(1+\sqrt{r'})^2}{r}}$ & $-$ \\
  \hline
  $3$ & $\displaystyle{\log\frac{1+3r'}{r}}$ & $\displaystyle{\frac{\exp(y)+3\sqrt{8+\exp(2y)}}{9+\exp(2y)}}$
 \\
  \hline
  $4$ & $\displaystyle{\log (1/r)}$ & $\displaystyle{\exp(-y)}$\\
  \hline
  $5$ &   $\displaystyle{\log \frac{1+r'}{r}}$ & $\displaystyle{\frac{2\exp(y)}{1+\exp(2y)}}$\\
  \hline
\end{tabular}
\vspace{0.3cm}
\caption{Lower bounds of $\mu$ and their inverses.% See \cite[pp.91-97]{avv}.
}
\label{Tab Lo b}
\end{table}
%~~~~~~~~~~~~~~~~~~~~~~~~~~~~~~~~~~~~~~~~~~~~~~~~
% Table for upper bounds and their inverses
%~~~~~~~~~~~~~~~~~~~~~~~~~~~~~~~~~~~~~~~~~~~~~~~~
%~~~~~~~~~~~~~~~~~~~~~~~~~~~~~~~~~~~~~~~~~~~~~~~~
\begin{table}[H]
\centering
\begin{tabular}{|c|c|c|}
    \hline
 $j$ & $v_j(r)$ &  $v_j^{-1}(y)$ \\
   \hline
    $1$ &$\displaystyle{\frac{\pi^2}{4 {\rm arth}\sqrt[4]{r}}}$ &
     $\displaystyle{{\rm th}^4 (\pi^2/4y)}$  \\
   \hline
    $2$ & $\displaystyle{\log \frac{2(1+r')}{r}}$ &
    $\displaystyle{\frac{4\exp(\max\{y,\log 2\})}{4+\exp(2\max\{y,\log 2\})}}$ \\
   \hline
    $3$ &  $\displaystyle{\log(4/r)}$ &  $\displaystyle{4\exp(-\max\{y,\log 4\})}$\\
   \hline
 \end{tabular}
\vspace{0.3cm}
\caption{Upper bounds of $\mu$ and their inverses. % See \cite[pp.91-97]{avv}.
}
\label{Tab up b}
\end{table}

\begin{table}[H]
\centering
\begin{tabular}{|c|c|c|}
    \hline
    Lower bounds & Upper bounds\\
    \hline
 $\displaystyle{\sqrt{1-{\rm th}^8(y)}}$ & $\displaystyle{{\rm th}^4 (\pi^2/4y)}$ \\
 \hline
  $\displaystyle{\frac{\exp(y)+3\sqrt{8+\exp(2y)}}{9+\exp(2y)}}$ &  $\displaystyle{\frac{4\exp(\max\{y,\log 2\})}{4+\exp(2\max\{y,\log 2\})}}$\\
  \hline
  $\displaystyle{\exp(-y)}$ & $\displaystyle{4\exp(-\max\{y,\log 4\})}$\\
  \hline
  $\displaystyle{\frac{2\exp(y)}{1+\exp(2y)}}$ & $-$\\
 \hline
\end{tabular}
\vspace{0.3cm}
\caption{Upper and lower bounds for $\mu^{-1}(y)$.
}
\label{Tab up lo b for InvMu}
\end{table}
\noindent
We ignore the inverse of $u_2$ in Table \ref{Tab Lo b} since it is a very complicated formula.
%~~~~~~~~~~~~~~~~~~~~~~~~~~~~~~~~~~~~~~~~~~~
%%%%%%%%%%%%%%%%%%%%%%%%%%%%%%
\begin{lemma}
  Assume that $u,v:(0,1)\rightarrow(0,\infty)$ are decreasing homeomorphisms with
\begin{equation}\label{ineq u mu v}
  u(r)<\mu(r)<v(r),\quad 0<r<1.
\end{equation}
 Then
  \begin{equation*}
    u^{-1}(v(r)/K) <\varphi_K(r) < v^{-1}(u(r)/K)
  \end{equation*}
  for all $K>1$ and $r\in(0,1)$.
\end{lemma}
%~~~~~~~~~~~~~~~~~~~~~~~~~~~~~~
\begin{proof}
Because $\mu$ is decreasing, and also by \eqref{ineq u mu v} we can obtain
\begin{equation*}
  \mu^{-1}(v(r)/K)<\varphi_K(r)=\mu^{-1}(\mu(r)/K)<\mu^{-1}(u(r)/K).
\end{equation*}
  It follows from $u^{-1}(y)<\mu^{-1}(y)$, $y>0$, that $u^{-1}(v(r)/K)<\mu^{-1}(v(r)/K)$. Also, since $\mu^{-1}(y)<v^{-1}(y)$, $y>0$, we get $\mu^{-1}(u(r)/K)<v^{-1}(u(r)/K)$ for all $K>1$ and $r\in(0,1)$. The proof is now complete.
\end{proof}
%~~~~~~~~~~~~~~~~~~~~~~~~~~~~~~~~~~~~~~~~
\begin{corollary}
  Let $u:(0,1)\rightarrow(0,\infty)$ and $v:(0,1]\rightarrow [c,\infty)$, $c>0$, be decreasing homeomorphisms which satisfy \eqref{ineq u mu v}. Then
  \begin{equation*}
    u^{-1}(v(r)/K) <\varphi_K(r) < v^{-1}(\max\{u(r)/K,c\})
\end{equation*}
  for all $K>1$ and $r\in(0,1)$.
\end{corollary}
%~~~~~~~~~~~~~~~~~~~~~~~~~~~~~~~~~~~~~~~~
\begin{example}
  Consider $u_1$ and $v_3$ as above. Since $v_3$ is a decreasing homeomorphism from $(0,1]$ onto $[\log 4,\infty)$, therefore,
    \begin{equation*}
    u_1^{-1}(v_3(r)/K) <\varphi_K(r) < v_3^{-1}(\max\{u_1(r)/K,\log 4\})
\end{equation*}
  for all $K>1$ and $r\in(0,1)$.
  \end{example}
%~~~~~~~~~~~~~~~~~~~~~~~~~~~~~~~~~~~~~~~~

%%%%%%%%%%%%%%%%%%%%%%%%%%%%%%
We recall the following lemma from \cite[p. 17]{avv}:
\begin{lemma}\label{lemAVVB1.45Thm}
Let $f$ be a decreasing homeomorphism from $(0, 1)$ onto $(0, \infty)$, and let $g, h$ be strictly decreasing continuous functions from $(0, 1)$ into $(0, \infty)$ with $h(0+) = \infty$ such that $g (r)< f(r) < h (r)$. Let $C > 1$ and $s =f^{-1}(f(r)/C)$. Then
\begin{equation*}
  g(r)> C g(s)\quad and \quad h(r)<C h(s),\quad 0<r<1,
\end{equation*}
if and only if $f(r)/g(r)$ and $h(r)/f(r)$ are strictly increasing on $(0,1)$. In particular,
if both $h^{-1}(h(r)/C)$ and $g^{-1}(g(r)/C)$ are defined, then
\begin{equation*}
  g^{-1}(g(r)/C)<s<h^{-1}(h(r)/C),\quad 0<r<1.
\end{equation*}
\end{lemma}
As an application of Lemma \ref{lemAVVB1.45Thm} we have:
\begin{lemma}
  Let $u_1(r)$ and $v_2(r)$ be defined as in Tables \ref{Tab Lo b} and \ref{Tab up b}, respectively, where $r\in(0,1)$. Then
  \begin{equation*}
    u_1^{-1}(u_1(r)/K)<\varphi_K(r)<v_2^{-1}(v_2(r)/K)
  \end{equation*}
  for all $K>1$ and $r\in(0,r_0)$, where $r_0\in(0,1)$ is such that both $u_1^{-1}(u_1(r)/K)$ and $v_2^{-1}(v_2(r)/K)$ are defined. Moreover, the first inequality is sharp in the sense that $ u_1^{-1}(u_1(r)/K)\to r$ as $K\rightarrow 1$.
\end{lemma}
%~~~~~~~~~~~~~~~~~~~~~~~~~~~~~~
\begin{proof}
  It follows from \eqref{bounds for mu-1} and \eqref{bounds for mu-2} that $u_1(r)<\mu(r)<v_2(r)$ for all $r\in(0,1)$. It is enough to show that both $\mu(r)/u_1(r)$ and $v_2(r)/\mu(r)$ are defined and strictly increasing on $(0,1)$. By using the first identity of \eqref{two iden for mu} and letting $u=2\sqrt{r}/(1+r)$, we have
  \begin{equation*}
    \frac{v_2(r)}{\mu(r)}\quad {\rm incr.}\, \Leftrightarrow \frac{v_2(u)}{\mu(u)}\quad {\rm incr.}\, \Leftrightarrow \frac{\log(2/\sqrt{r})}{\mu(2\sqrt{r}/(1+r))}\quad {\rm incr.}\, \Leftrightarrow \frac{\frac{1}{2}\log(4/r)}{\frac{1}{2}\mu(r)}\quad {\rm incr.}
  \end{equation*}
  which is valid by \cite[Theorem 2.16(2)]{avv}. It follows also from \cite[Theorem 5.13(6)]{avv} that $\mu(r)/u_1(r)$ is strictly increasing from $(0,1)$ onto $(1,\infty)$. For illustration, see Figure \ref{Fig-muu1v2mu}. The proof is now complete.
\end{proof}
  %*************************************************
\begin{figure}[!ht]
\centering
\subfigure[]{
\includegraphics[width=5.5cm]{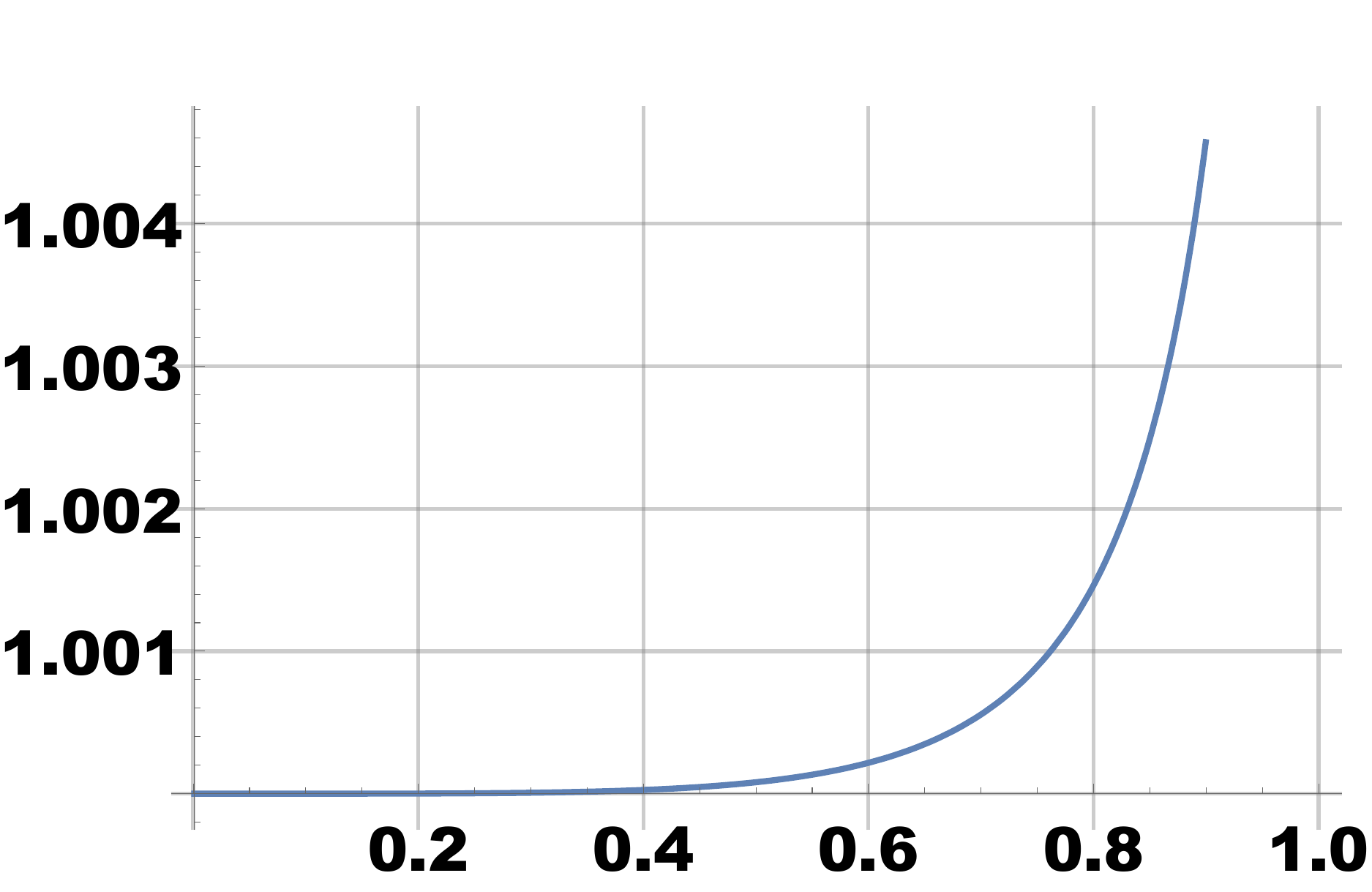}
    \label{fig:subfig1}
}
\hspace*{7mm}
\subfigure[]{
\includegraphics[width=5.5cm]{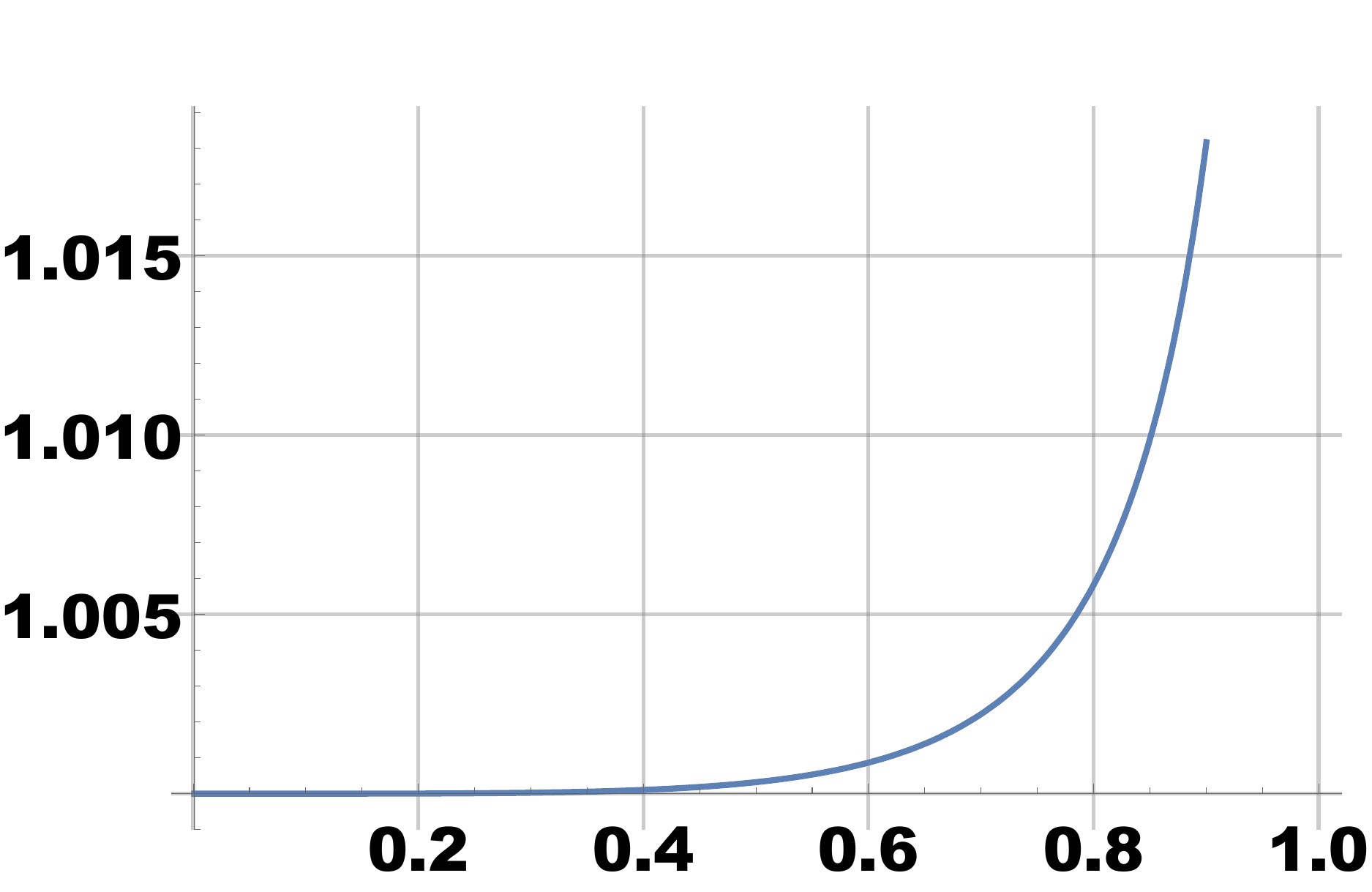}
    \label{fig:subfig2}
}
\caption[The graph of $\mu(r)/u_1(r)$ and $v_2(r)/\mu(r)$]
{\subref{fig:subfig1}: The graph of $\mu(r)/u_1(r)$, where $0<r<1$
 \subref{fig:subfig2}: The graph of $v_2(r)/\mu(r)$, where $0<r<1$.}
 \label{Fig-muu1v2mu}
\end{figure}
 %Figures--------------------------------------
%%%%%%%%%%%%%%%%%%%%%%%%%%%%%%
\begin{remark}
  It is worth mentioning that the new upper and lower bounds for $\varphi_K$ are more accurate than those bounds in \eqref{esti varphi}. Indeed,
    \begin{equation*}
    r^{1/K}<u_1^{-1}(u_1(r)/K)<\varphi_K(r)<v_2^{-1}(v_2(r)/K)<4^{1-1/K}r^{1/K}
  \end{equation*}
  for all $K>1$ and $r\in(0,1)$.
\end{remark}
%%%%%%%%%%%%%%%%%%%%%%%%%%%%%%
\section{Landen Approximations}\label{sec3}

The three functions  $\mu, \mu^{-1}, \varphi_K$ were extensively investigated in
\cite{avv}, with computer implementations in languages, Mathematica, MATLAB, C
on the accompanying diskette. Here our goal is to show that for a large range
of the arguments we obtain results with accuracy to those in \cite{avv},
now only using Mathematica. The methods applied in \cite{avv} for the numerical
evaluation of $\K$ and $\mu$ were based on the arithmetic-geometric meanwhile for $\mu^{-1}$ and $\varphi_K$ a Newton iteration was used. Here we show that
the Landen sequences yield approximations with errors close to machine
epsilon agreement with the results of
\cite{avv} when the recursion level is moderate, 4 or 5.

Our starting point is the following lemma.

\begin{lemma}\label{muMonot} (1) The function $\mu(r)+ \log r$ is a monotone
decreasing function from $(0,1)$ onto $(0,\log 4).$

\noindent
(2) For $0<r<1, p \in \mathbb{Z}$
\begin{equation*}
2^{-p} \log \frac{1}{L(r,-p)}< \mu(r) <  2^{-p} \log \frac{4}{L(r,-p)},
\end{equation*}
and, in particular, for $p=1$ we have
\begin{equation*}
 %\log \frac{1+r'}{r}< \mu(r) <   \log \frac{2(1+r')}{r}.
 u_5(r)<\mu(r)<v_2(r),
\end{equation*}
where $u_5$ and $v_2$ are defined as in Tables \ref{Tab Lo b} and \ref{Tab up b}, respectively.
\end{lemma}
\begin{proof} (1) This is well-known, see \cite[p. 84, Thm 5.13]{avv}.

\noindent
(2) The proof follows from \eqref{mu(r)=2 p mu} and part (1).
\end{proof}
%~~~~~~~~~~~~~~~~~~~~~~~~~~~~~~~~~~~~~~~~~~~~
The upper bound of Lemma \ref{muMonot}(2) with $p=4$ seems to be very precise. Figure \ref{Fig-mu-ub} shows the difference $\mu(r)-2^{-4}\log(4/L(r,-4))$, where $r\in(0,1)$ and $\mu(r)$ is computed using the {\tt cip.nb} file from \cite{avv}.
\begin{figure}
  \centering
  \includegraphics[width=3in]{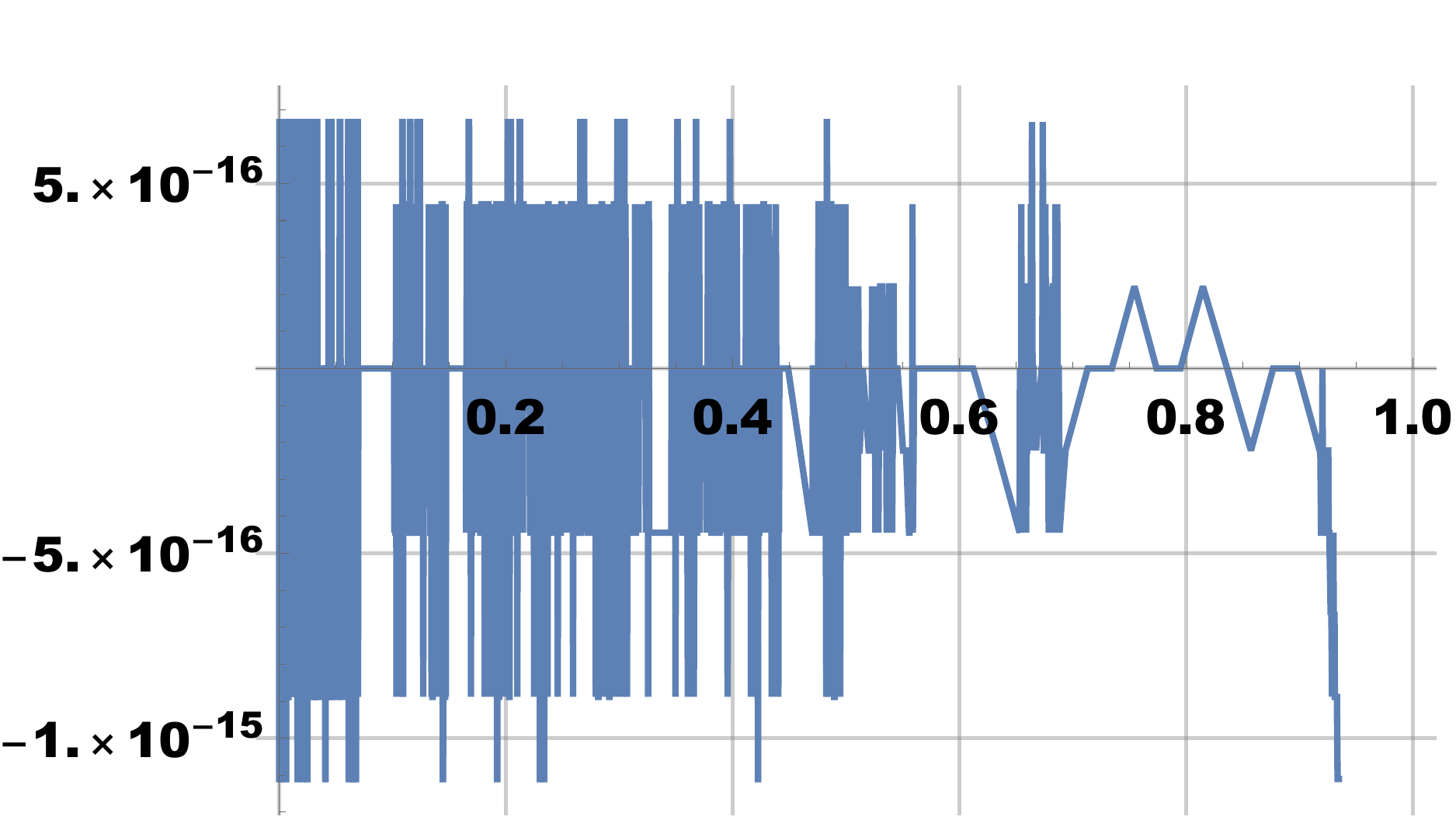}
  \caption{The difference between $\mu(r)$ and the upper bound of  Lemma \ref{muMonot}(2) with $p=4$ and $r\in(0,1)$.}\label{Fig-mu-ub}
\end{figure}
%~~~~~~~~~~~~~~~~~~~~~~~~~~~~~~~~~~~~~~~~~~~~
\begin{proposition}\label{prop lim for 2 p u}
  Let $u,v:(0,1)\rightarrow(0,\infty)$ be continuous functions with $|u(r)-v(r)|<M$ for some constant $M$ and for all $r\in(0,1)$. Also let $u(r)<\mu(r)<v(r)$ for all $r\in(0,1)$. Then
\begin{equation*}
  \lim_{p\rightarrow -\infty}2^p u(L(r,p))=\lim_{p\rightarrow -\infty}2^p v(L(r,p))=\mu(r).
\end{equation*}
\end{proposition}
\begin{proof}
  By \eqref{mu(r)=2 p mu}
\begin{equation*}
  2^p u(L(r,p))< \mu(r)=2^p\mu(L(r,p))< 2^p v(L(r,p))
\end{equation*}
which implies that
\begin{equation*}
  |2^p v(L(r,p))-2^p u(L(r,p))|<2^p M,
\end{equation*}
concluding the proof.
\end{proof}

\begin{remark}
  By \eqref{bounds for mu-1} we can apply Proposition \ref{prop lim for 2 p u}, for example, with $u_2$ and $v_2$, where $u_2$ and $v_2$ are defined as in Tables \ref{Tab Lo b} and \ref{Tab up b}, respectively.
\end{remark}
%~~~~~~~~~~~~~~~~~~~~~~~~~~~~~~~~~~~~~~~~~~~~
%~~~~~~~~~~~~~~~~~~~~~~~~~~~~~~~~~~~~~~~~~~~~
Below, we will find the best and simplest approximation for $\mu^{-1}(y)$, based on the Landen transformation.
%~~~~~~~~~~~~~~~~~~~~~~~~~~~~~~~~~~~~~~~~~~~~
\begin{proposition}\label{prop ineq for L}
    Let $u,v:(0,1)\rightarrow(0,\infty)$ be decreasing homeomorphism with $u(r)<\mu(r)<v(r)$ for all $r\in(0,1)$. Then for $y>0$ and $r=\mu^{-1}(y)$ we have
\begin{equation*}
  L(u^{-1}(2^{-p}y),-p)<r=\mu^{-1}(y)<L(v^{-1}(2^{-p}y),-p).
\end{equation*}
\end{proposition}
%~~~~~~~~~~~~~~~~~~~~~~~~~~~~~~~~~~~~~~~~~~~~
\begin{proof}
  By \eqref{mu(r)=2 p mu}
\begin{equation*}
  y=\mu(r)=2^p \mu(L(r,p))<2^p v(L(r,p))
\end{equation*}
and because $v$ is decreasing $L(r,p)<v^{-1}(2^{-p}y)$. Hence $r < L(v^{-1}(2^{-p}y),-p)$. The proof of the lower bound is similar, so we omit the details.
\end{proof}
%~~~~~~~~~~~~~~~~~~~~~~~~~~~~~~~~~~~~~~~~~~~~
\begin{remark}\label{rem application}
We know by \eqref{bounds for mu-2} that
\begin{equation*}
  u_1(r)<\mu(r)<v_1(r)
\end{equation*}
for all $r\in(0,1)$, where $u_1$ and $v_1$ are defined as in Tables \ref{Tab Lo b} and \ref{Tab up b}, respectively.
It is easy to see that $u_1$ and $v_1$ are a homeomorphism of $(0,1)$ onto $(0,\infty)$.
Applying Proposition \ref{prop ineq for L} with $u_1(r)$, $v_1(r)$, and their inverses we obtain
\begin{equation*}
\mu^{-1}(y)\approx L\left(u_1^{-1}(2^{-p}y),-p\right)=:f_1(y,p)
\end{equation*}
and
\begin{equation*}
\mu^{-1}(y)\approx L\left(v_1^{-1}(2^{-p}y),-p\right)=:g_1(y,p).
\end{equation*}
%~~~~~~~~~~~~~~~~~~~~~~~~~~~~~~~~~~~~~~~~~~~~~~~~~~
%~~~~~~~~~~~~~~~~~~~~~~~~~~~~~~~~~~~~~~~~~~~~~~~~~~
It also follows from \eqref{bounds for mu-1} that the following inequalities
\begin{equation*}
  u_3(r)<\mu(r)<v_2(r),
\end{equation*}
hold true for $r\in(0,1)$, where $u_3$ and $v_2$ are a homeomorphism of $(0,1)$ onto $(0,\infty)$ and $(\log 2,\infty)$, respectively.
If we apply Proposition \ref{prop ineq for L} for $u_3$ and $v_2$, we get
\begin{equation*}
\mu^{-1}(y)\approx L\left(u_3^{-1}(2^{-p}y),-p\right)=:f_2(y,p)
\end{equation*}
and
\begin{equation*}
\mu^{-1}(y)\approx L\left(v_2^{-1}(2^{-p}y),-p\right)=:g_2(y,p).
\end{equation*}
Finally, applying
\begin{equation*}
  u_4(r)<\mu(r)<v_3(r)
\end{equation*}
and applying Proposition \ref{prop ineq for L} with $u_4$ and $v_3$ (which are a homeomorphism of $(0,1)$ onto $(0,\infty)$ and $(\log4,\infty)$, respectively) we obtain
\begin{equation*}
  \mu^{-1}(y)\approx L\left(u_4^{-1}(2^{-p}y),-p\right)=:f_3(y,p)
\end{equation*}
and
\begin{equation*}
  \mu^{-1}(y)\approx L\left(v_3^{-1}(2^{-p}y),-p\right)=:g_3(y,p).
\end{equation*}
Computational results for some values $y$ in the range $(0.2, 20)$ are summarized in Table \ref{Tab4}. Only the cases $p=-4,-5$ are taken into account in this table. When $p=\ldots,-6,-3,-2,-1,\ldots$ the error is large.
%~~~~~~~~~~~~~~~~~~~~~~~~~~~~~~~~~~~~~~~~~~~~~~~~
\begin{table}[H]
\centering
\begin{tabular}{|c|c|c|}
    \hline
 $y$ & $\mu^{-1}(y)-g_2(y,-4)$ &  $\mu^{-1}(y)-g_3(y,-5)$ \\
  \hline
 $0.5$ &  $0$ &  $0$\\
 \hline
 $1.5$ &  $-2.22045\times10^{-16}$ &  $-2.22045\times10^{-16}$\\
   \hline
 $2.5$ &  $1.11022\times10^{-16}$ &  $1.11022\times10^{-16}$\\
   \hline
    $3.5$ &  $-1.38778\times10^{-17}$ &  $-1.38778\times10^{-17}$\\
   \hline
    $4.5$ &  $2.08167\times10^{-17}$ &  $2.08167\times10^{-17}$\\
   \hline
    $5.5$ &  $0$ &  $0$\\
   \hline
    $6.5$ &  $0$ &  $0$\\
   \hline
    $7.5$ &  $-1.30104\times10^{-18}$ &  $-1.30104\times10^{-18}$\\
   \hline
    $8.5$ &  $-4.33681\times10^{-19}$ &  $-4.33681\times10^{-19}$\\
   \hline
       $9.5$ &  $-7.58942\times10^{-19}$ &  $-7.58942\times10^{-19}$\\
   \hline
       $10.5$ &  $-2.1684\times10^{-19}$ &  $-2.1684\times10^{-19}$\\
   \hline
       $11.5$ &  $3.38813\times10^{-20}$ &  $3.38813\times10^{-20}$\\
   \hline
       $12.5$ &  $-1.18585\times10^{-20}$ &  $-1.18585\times10^{-20}$\\
   \hline
       $13.5$ &  $-3.38813\times10^{-21}$ &  $-3.38813\times10^{-21}$\\
   \hline
       $14.5$ &  $-4.23516\times10^{-22}$ &  $0$\\
   \hline
       $15.5$ &  $-1.90582\times10^{-21}$ &  $-1.90582\times10^{-21}$\\
   \hline
       $16.5$ &  $-5.82335\times10^{-22}$ &  $-5.82335\times10^{-22}$\\
   \hline
          $17.5$ &  $2.24993\times10^{-22}$ &  $2.24993\times10^{-22}$\\
   \hline
          $18.5$ &  $7.27919\times10^{-23}$ &  $7.27919\times10^{-23}$\\
   \hline
          $19.5$ &  $1.65436\times10^{-23}$ &  $1.65436\times10^{-23}$\\
   \hline
 \end{tabular}
\vspace{0.3cm}
\caption{The error between $\mu^{-1}(y)$ and $g_2(y,-4)$, and $\mu^{-1}(y)$ and $g_3(y,-5)$ for some values in range $y\in(0.2,20)$. For the computation of $\mu^{-1}(y)$ we have used the Mathematica "cip.m" file from \cite[Appendix B]{avv}.}
\label{Tab4}
\end{table}
Computer experiments show that $g_2(y,-4)$ and $g_3(y,-5)$ are the best approximations for $\mu^{-1}(y)$. We note that $\mu^{-1}(y)-g_2(y,-4)$ and $\mu^{-1}(y)-g_3(y,-5)$ have an error value of order $10^{-14},\ldots,10^{-24}$ in the interval $(0.2,20)$, see Figure \ref{Fig-InvMu-g23}. For the computation of $\mu^{-1}(y)$ we use the Mathematica "cip.m" file from \cite[Appendix B]{avv}.
  %*************************************************
\begin{figure}[!ht]
\centering
\subfigure[]{
\includegraphics[width=5.5cm]{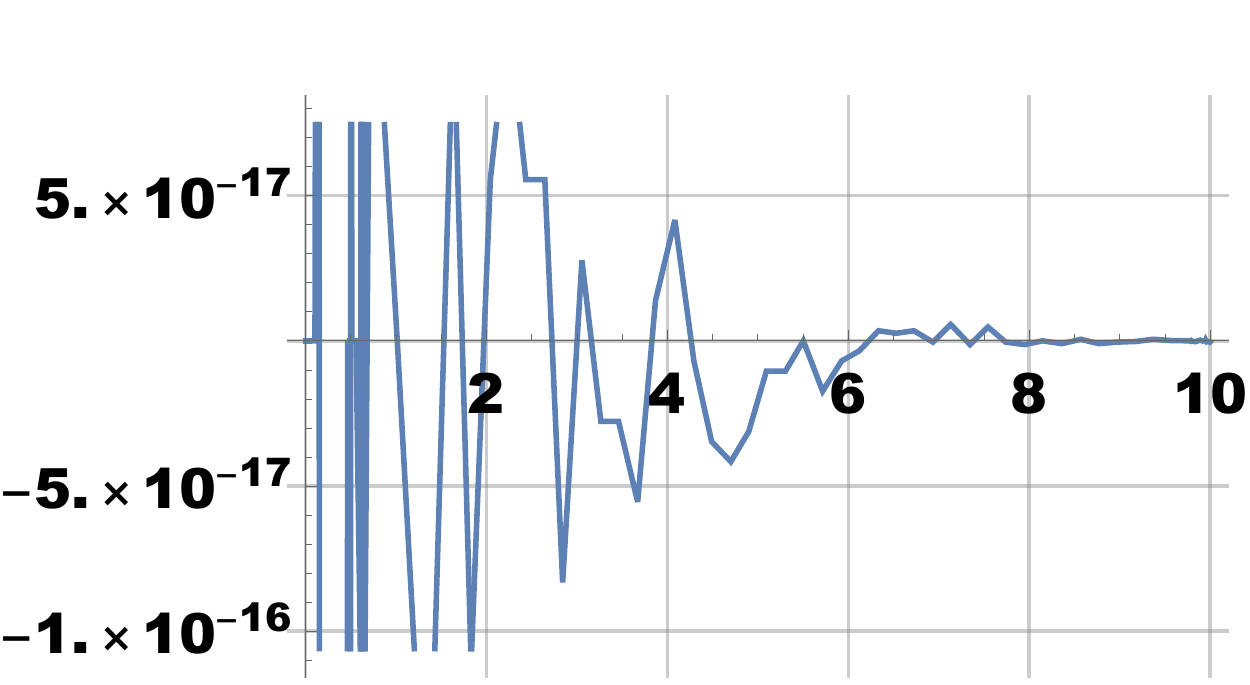}
    \label{fig:subfigInvMu-g2}
}
\hspace*{7mm}
\subfigure[]{
\includegraphics[width=5.5cm]{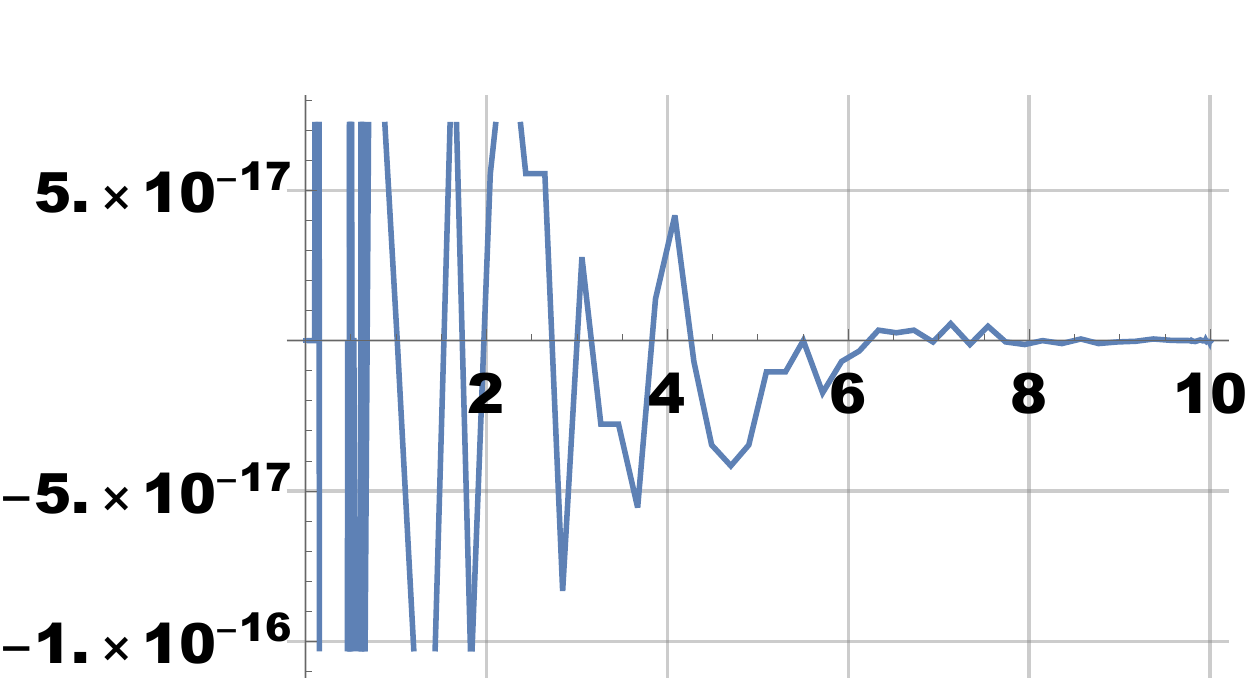}
    \label{fig:subfigInvMu-g3}
}
\caption[The graph of error between $\mu^{-1}(y)$ and $g_{2,3}$]
{\subref{fig:subfig1}: The graph of $\mu^{-1}(y)-g_2(y,-4)$, where $0<y<10$
 \subref{fig:subfig2}: The graph of $\mu^{-1}(y)-g_3(y,-5)$, where $0<y<10$.}
 \label{Fig-InvMu-g23}
\end{figure}
 %Figures--------------------------------------
\end{remark}
%~~~~~~~~~~~~~~~~~~~~~~~~~~~~~~~~~~~~~~~~~~~~
%~~~~~~~~~~~~~~~~~~~~~~~~~~~~~~~~~~~~~~~~~~~~
%~~~~~~~~~~~~~~~~~~~~~~~~~~~~~~~~~~~~~~~~~~~~
\begin{table}[H]
\centering
\begin{tabular}{|c|c|c|}
    \hline
 $p$ &   $\mu^{-1}(y)$ & $\mu(r)$ \\
   \hline
 $1$ & $\displaystyle{\frac{4\sqrt{\exp(-\max\{2r,\log4\})}}{1+4\exp(-\max\{2r,\log4\})}}$ &
  $\displaystyle{\log \frac{2(1+r')}{r}}$
 \\
\hline
$2$ & $\displaystyle{\frac{4\sqrt{\frac{\sqrt{\exp(-\max\{4r,\log4\})}}{1+4\exp(-\max\{4r,\log4\})}}}
{1+\frac{4\sqrt{\exp(-\max\{4r,\log4\})}}{1+4\exp(-\max\{4r,\log4\})}}}$
&$\displaystyle{\frac{1}{4} \log\frac{4\left(1+r'\right)^4
\left(1+\sqrt{w}\right)^2}{r^4}}$\\
\hline
\end{tabular}
\vspace{0.3cm}
\caption{Two steps of Landen approximations of $\mu^{-1}$ by $g_3(y,-p))$, and $\mu$ by Lemma \ref{muMonot}(2).}
\label{PhiFunc}
\end{table}

\begin{nonsec}{\bf The special function $\varphi_K(r)$.}\label{varphi fun sec}
In the study of H\"{o}lder continuity of quasiconformal mappings of the plane, the special function $\varphi_K(r)$ defined as \eqref{varphi}
has an important role.
Based on Proposition \ref{prop lim for 2 p u} and Remark  \ref{rem application} we study the approximation \cite[Theorem 5.43]{avv}
\begin{equation}\label{varphi approx}
  \varphi_K(r)=L(\varphi_K(L(r,-p)),p)\approx L\left(4^{1-1/K} L(r,-p)^{1/K},p\right)
  =:L_\varphi(K,r,p)
\end{equation}
for various values of $p$ with $4^{1-1/K} L(r,-p)^{1/K}<1$. Table \ref{PhiApp} shows a structural formula for $L_\varphi(K,r,p)$, where $p=0,1$.
\begin{table}[H]
\centering
\begin{tabular}{|c|c|c|}
    \hline
 $p$ &   $L_\varphi(K,r,p)$&  c\\
   \hline
   $0$ & $4^{1-1/K} r^{1/K}$ & $-$\\
   \hline
     $1$ & ${\displaystyle{\frac{2\sqrt{4^{1-1/K}\,\, c^{1/K}}}{1+4^{1-1/K}\,\, c^{1/K}}}}$ & $\displaystyle{\left(\frac{r}{1+\sqrt{1-r^2}}\right)^2}$\\
\hline
\end{tabular}
\vspace{0.3cm}
\caption{The function $L_\varphi(K,r,p)$ for $p=0,1$.}
\label{PhiApp}
\end{table}
Here we note that $L_\varphi(K,r,p)$ is a majorant for the function $\varphi_K(r)$ when $4^{1-1/K} L(r,-p)^{1/K}<1$.
We also study the following approximation by applying Remark \ref{rem application} and Lemma \ref{muMonot}(2)
\begin{equation}\label{varphi approx 2}
     \varphi_K(r)\approx g_3\left(2^{-p}\log(4/L(r,-p))/K,-5\right)=:LM(K,r,p),
\end{equation}
where $K>1$, and $r\in(0,1)$. Computer experiments show that $LM(K,r,5)$ is the best approximation for $\varphi_K(r)$, see Figure \ref{Fig-LM3D}.
\begin{figure}
  \centering
  \includegraphics[width=10cm]{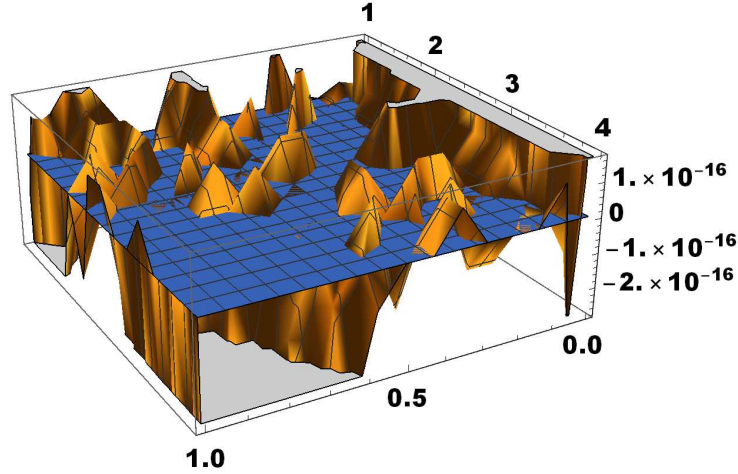}
  \caption{The 3D plot of $\varphi_K(r)-LM(K,r,5)$ for $1<K<4$ and $0<r<1$.}\label{Fig-LM3D}
\end{figure}
\end{nonsec}
%~~~~~~~~~~~~~~~~~~~~~~~~~~~~~~~~~~~~~~~~~~~¨
\begin{nonsec}{\bf Remark.}\label{rmkLPhi} The next few lines of Mathematica code
\begin{verbatim}
L[s_, p_] :=  Module[{j = 0, y = s},
   While[((j < Abs[p])), If[p < 0, y = (y / (1 + Sqrt[1 - y^2]))^2,
     y = 2 * Sqrt[y] / (1 + y) ]; j++]; y];
LPhi[K_, r_, p_] := L[4 * Exp[(1 / K) * Log[L[r, -p] / 4]], p];
\end{verbatim}
%\noindent
define an approximation for $\varphi_K$. This function satisfies
\begin{equation}\label{basic}
 4^{1 - 1/K} r^{1/K} > {\tt LPhi[K, r, 1]}
 \end{equation}
as we see using the command
\begin{verbatim}
 Plot3D[{0, 4^(1 - 1/K) r^(1/K) - LPhi[K, r, 1]}, {r, 0, 1}, {K, 1, 3}]
 \end{verbatim}
The LHS function of \eqref{basic} here is a majorant for  $\varphi_K, K>1,$ i.e.
  $\varphi_K(r) < 4^{1 - 1/K} r^{1/K}$ for $ K>1, r \in (0,1)$ by
  \cite[Thm 9.32]{hkv}.

The RHS function {\tt LPhi[K, r, 1]} of  \eqref{basic} is not well-defined, e.g., for $K=2$ and $r=0.9$, because
\begin{equation*}
    4^{1-1/K}L(0.9,-1)^{1/2}=1.25358>1.
\end{equation*}
\end{nonsec}

%~~~~~~~~~~~~~~~~~~~~~~~~~~~~~~~~~~~~~~~~~~~~
\begin{nonsec}{\bf Conclusion.}\label{canclusion}
For $K\in(1,20)$ and $r\in(0,1)$ the approximations \eqref{varphi approx} and
\eqref{varphi approx 2} with $p=5$ yield maximal error of the order $10^{-14}$.
The reported error is based on the identity \eqref{phiPyth}.
The approximation \eqref{varphi approx} based only on
the Landen transformation is remarkably simple and precise,
as it makes no use of elliptic integrals.
One could also use this identity \eqref{varphi2} to test the above algorithm.
\end{nonsec}
%~~~~~~~~~~~~~~~~~~~~~~~~~~~~~~~~~~~~~~~~~~~~

%~~~~~~~~~~~~~~~~~~~~~~~~~~~~~~~~~~~~~~~~~~~~
\begin{nonsec}{\bf Some open problems.}\label{open p}
Computational experiments have led us to formulate the following questions:\\
{\bf (1)} Let $L_\varphi(K,r,p)$ be defined as in \eqref{varphi approx}. Then
\begin{equation*}
  L_\varphi(K,r,p)\leq 4^{1-{1}/{K}} r^{{1}/{K}}
\end{equation*}
for $1<K<4.6$, $r\in(0,0.7]$, and $p=0,1,2,\ldots$.

\noindent
{\bf Motivation.} Considering that $p=0$ is obvious, we may assume that $p=1,2,3,\ldots$.
Since $L(\cdot, p):(0, 1) \rightarrow (0, 1)$ is an increasing homeomorphism, we are looking for $r\in(0,1)$ and $K>1$ such that $4^{1-1/K} L(r,-p)^{1/K}<1$. Computer experiments show that $4^{1-1/K} L(r,-p)^{1/K}<1$ holds true for all $r\in(0,0.7]$, $1<K<4.6$, and $p=1,2,\ldots$.

\noindent
{\bf (2)}  Remark \ref{rmkLPhi} only deals with the case $p=1$. What about $p=2$?
Can we find some pair of functions $u,v$ where $u$ is a minorant of $\mu$ and
$v$ a majorant of $\mu$ such that the corresponding $L_\varphi(K,r,1)$ would be
a majorant of $\varphi_K$?
\end{nonsec}
%~~~~~~~~~~~~~~~~~~~~~~~~~~~~~~~~~~~~~~~~~~~~
\noindent
{\bf Funding.}
The first author received financial support provided by the Doctoral Programme (formerly MATTI and now EXACTUS) of the Department of Mathematics and Statistics of the University of Turku. The research of the second author was supported by the Finnish Cultural Foundation.

%~~~~~~~~~~~~~~~~~~~~~~~~~~~~~~~~~~~~~~~~~~~~
\vspace{1cm}
\noindent
{\bf Acknowledgment.}
The third author had an opportunity to work at the Institute of Mathematics in Novosibirsk for the period March-May 1981 as a guest of Prof. Yu. G. Reshetnyak, a pioneer of the theory of quasiregular mappings, Bull. Amer. Math. Soc. (N.S.) 24 (1991), no. 2, 408--415.
The third author is happy to acknowledge the strong influence of Prof. Reshetnyak's work on his whole research career.

%~~~~~~~~~~~~~~~~~~~~~~~~~~~~~~~~~~~~~~~~~~~~
\vspace{1cm}
\noindent
{\bf Orcid}\\
\noindent
{\tt Rahim Kargar ORCID ID: \url{http://orcid.org/0000-0003-1029-5386}}\\
{\tt Oona Rainio ORCID ID: \url{http://orcid.org/0000-0002-7775-7656}}\\
{\tt Matti Vuorinen ORCID ID: \url{http://orcid.org/0000-0002-1734-8228}}

%~~~~~~~~~~~~~~~~~~~~~~~~~~~~~~~~~~~~~~~~~~~~

%~~~~~~~~~~~~~~~~~~~~~~~~~~~~~~~~~~~~~~~~~~~~
%~~~~~~~~~~~~~~~~~~~~~~~~~~~~~~~~~~~~~~~~~~~~
%~~~~~~~~~~~~~~~~~~~~~~~~~~~~~~~~~~~~~~~~~~~~
%~~~~~~~~~~~~~~~~~~~~~~~~~~~~~~~~~~~~~~~~~~~~
%~~~~~~~~~~~~~~~~~~~~~~~~~~~~~~~~~~~~~~~~~~~~
%~~~~~~~~~~~~~~~~~~~~~~~~~~~~~~~~~~~~~~~~~~~~
%~~~~~~~~~~~~~~~~~~~~~~~~~~~~~~~~~~~~~~~~~~~~
%~~~~~~~~~~~~~~~~~~~~~~~~~~~~~~~~~~~~~~~~~~~~
%~~~~~~~~~~~~~~~~~~~~~~~~~~~~~~~~~~~~~~~~~~~~
%~~~~~~~~~~~~~~~~~~~~~~~~~~~~~~~~~~~~~~~~~~~~
%~~~~~~~~~~~~~~~~~~~~~~~~~~~~~~~~~~~~~~~~~~~~
%~~~~~~~~~~~~~~~~~~~~~~~~~~~~~~~~~~~~~~~~~~~~

%%%%%%%%%%%%%%%%%%%%%%%%%%%%%%

\def\cprime{$'$} \def\cprime{$'$} \def\cprime{$'$}
\providecommand{\bysame}{\leavevmode\hbox to3em{\hrulefill}\thinspace}
\providecommand{\MR}{\relax\ifhmode\unskip\space\fi MR }
\providecommand{\MRhref}[2]{%
  \href{http://www.ams.org/mathscinet-getitem?mr=#1}{#2}
}
\providecommand{\href}[2]{#2}

\end{document}